\documentclass[a4paper, 11pt]{amsart}
\usepackage{amsmath, amssymb}
\usepackage{amsfonts}
\usepackage{mathrsfs}
\usepackage[arrow,matrix,curve,cmtip,ps]{xy}

\usepackage{amsthm}

\usepackage{color}
\usepackage{amsfonts}
\usepackage{amscd}
\usepackage{comment}
\usepackage{epsfig}

\usepackage[colorlinks=true, pdfstartview=FitV, linkcolor=red, citecolor=blue, urlcolor=green]{hyperref}
\usepackage{tikz,etex,bbm,xspace,hyperref}
\newcommand{\arxiv}[1]{\href{http://arxiv.org/abs/#1}{\tt arXiv:\nolinkurl{#1}}}

\allowdisplaybreaks

\newtheorem{theorem}{Theorem}[section]
\newtheorem{lemma}[theorem]{Lemma}
\newtheorem{proposition}[theorem]{Proposition}

\newtheorem*{theorem*}{Theorem}
\theoremstyle{remark}
\newtheorem{remark}[theorem]{Remark}
\newtheorem{definition}[theorem]{Definition}

\numberwithin{equation}{section}

\newcommand{\ci}[1]{_{ {}_{\scriptstyle #1}}}
\newcommand{\ti}[1]{_{\scriptstyle \text{\rm #1}}}

\newcounter{vremennyj}

\renewcommand{\eqref}[1]{Equation (\ref{#1})}
         
\begin{document}
\title{A note on martingale Hardy spaces}

\author{Jingguo Lai}

\address{Department of Mathematics \\ Brown University \\ Providence, RI 02912 \\ USA}
\email{jglai@math.brown.edu}





\begin{abstract}

We investigate the relation between Carleson sequence and balayage, and use this to give an easy proof of the equivalence of the $L^1$-norms of the maximal function and the square function in non-honogeneous martingale settings.   

\end{abstract}

\maketitle

\section{Introduction and the main theorem}
In this note, we attempt to give an easy proof of the celebrated theorem:
\begin{center}
\emph{The $L^1$-norms of the maximal function and square function are equivalent.}
\end{center}

\par Throughout the note, we will work on the real line $\mathbb{R}$ under Lebesgue measure, and assume our $\sigma$-algebras are generated by disjoint intervals. The notation $\lesssim$, $\gtrsim$ stand for one-sided estimates up to an absolute constant, and the notation $\approx$ stands for two-sided estimates up to an absolute constant. We will introduce the basic set-up following from \cite{ST}.

\begin{definition}
A \emph{lattice} $\mathcal{L}$ is a collection of non-trivial finite intervals of $\mathbb{R}$ with the following properties
\begin{enumerate}
\item $\mathcal{L}$ is a union of \emph{generations} $\mathcal{L}_k, k\in\mathbb{Z}$, where each generation is a collection of disjoint intervals, covering $\mathbb{R}$.
\item For each $k\in\mathbb{Z}$, the covering $\mathcal{L}_{k+1}$ is a finite refinement of the covering $\mathcal{L}_k$, i.e. each interval $I\in\mathcal{L}_k$ is a finite union of disjoint intervals  $J\in\mathcal{L}_{k+1}$. We allow the situation where there is only one such interval $J$, i.e. $J=I$; this means that $I\in\mathcal{L}_k$ also belongs to the generation $\mathcal{L}_{k+1}$.
\end{enumerate}
We say a lattice $\mathcal{L}$ is \emph{homogenous} if
\begin{enumerate}
\item Each interval $I\in\mathcal{L}_k$ is a union of \emph{at most} $r$ intervals $J\in\mathcal{L}_{k+1}$.
\item There exists a constant $K<\infty$ such that $\frac{|I|}{|J|}\leq K$ for every $I\in\mathcal{L}_k$ and every $J\in\mathcal{L}_{k+1}, J\subseteq I$.
\end{enumerate}
The standard dyadic lattice $\mathcal{D}=\{([0,1)+j)2^k:j,k\in\mathbb{Z}\}$ is an example of homogenous lattice.
\end{definition}

\begin{remark}
Let $\mathfrak{A}_k$ be the $\sigma$-algebra generated by $\mathcal{L}_k$, i.e. countable unions of intervals in $\mathcal{L}_k$. Let $\mathfrak{A}_{-\infty}$ be the largest $\sigma$-algebra contained in all $\mathfrak{A}_k, k\in\mathbb{Z}$, i.e.  $\mathfrak{A}_{-\infty}=\bigcap_{k\in\mathbb{Z}}\mathfrak{A}_k$, and let $\mathfrak{A}_\infty$ be the smallest $\sigma$-algebra containing all $\mathfrak{A}_k, k\in\mathbb{Z}$. The structures of $\sigma$-algebras $\mathfrak{A}_{-\infty}$ and $\mathfrak{A}_\infty$ can be described as follows\\
$\mathfrak{A}_{-\infty}$ is the $\sigma$-algebra generated by all the intervals $I$ of form
$$I=\bigcup_{k\in\mathbb{Z}}I_k, ~\textup{where}~ I_k\in\mathcal{L}_k, I_k\subseteq I_{k-1}.$$
Note that $\mathbb{R}$ is a disjoint union of such intervals $I$ and at most countably many points (we might need to add left endpoints to the intervals $I$). Let us denote the collection of such intervals $I$ by $\mathfrak{A}_{-\infty}^0$. Define
$$\mathfrak{A}_{-\infty}^{0,\textup{fin}}=\{I\in\mathfrak{A}_{-\infty}^0:|I|<\infty\};$$
'fin' here is to remind that the set consists of intervals of finite measure.\\
Instead of describing $\mathfrak{A}_\infty$, let us describe the corresponding measurable functions. \emph{Namely, a function $f(x)$ is $\mathfrak{A}_\infty$-measurable, if it is Borel measurable and it is constant on intervals $I$}.
$$I=\bigcap_{k\in\mathbb{Z}}I_k, ~\textup{where}~ I_k\in\mathcal{L}_k, I_k\subseteq I_{k-1}.$$
Clearly, such intervals $I$ do not intersect, so there can only be countably many of them.\\
In this note, we always assume that all functions are $\mathfrak{A}_\infty$-measurable.
\end{remark}

\begin{definition}
The \emph{averaging operator} $\mathbb{E}_{\ci{I}}$  is defined to be 
$$\mathbb{E}_{\ci{I}}f=\langle f\rangle_{\ci{I}}\mathbf{1}_{\ci{I}}=\left(\frac{1}{|I|}\int_If\right)\mathbf{1}_{\ci{I}},$$ 
where $\mathbf{1}_{\ci{I}}$ is the indicator function of the interval $I$. \\ 
Let $\mathbb{E}_k$ denote the \emph{conditional expectation}
$$\mathbb{E}_kf=\mathbb{E}\left[f|\mathfrak{A}_k\right]=\sum_{I\in\mathcal{L}_k}\mathbb{E}_{\ci{I}}f.$$
For an interval $I\in\mathcal{L}$, let rk($I$) be the \emph{rank} of the interval $I$, i.e. the largest number $k$ such that $I\in\mathcal{L}_k$.\\
For an interval $I\in\mathcal{L}, \textrm{rk}(I)=k$, a \emph{child} of $I$ is an interval $J\in\mathcal{L}_{k+1}$ such that $J\subseteq I$ (actually we can write $J\subsetneqq I$). The colletion of all children of $I$ is denoted by child($I$). Correspondingly, $I$ is called the \emph{parent} of $J$.\\
The \emph{difference operator} $\Delta_{\ci{I}}$ is defined to be
$$\Delta_{\ci{I}}f=-\mathbb{E}_{\ci{I}}f+\sum_{J\in\textup{child}(I)}\mathbb{E}_{\ci{J}}f.$$
Let $\Delta_k$ denote the \emph{martingale difference}
$$\Delta_kf=\mathbb{E}_kf-\mathbb{E}_{k-1}f=\sum_{I\in\mathcal{L}, \textup{rk}(I)=k-1}\Delta_{\ci{I}}f.$$
\end{definition}

To motivate the definition of the square function as well as for our later purpose, let us prove the following proposition:

\begin{proposition}
For a function $f(x)\in L^p$, where $p > 1$, 
\begin{align}
f=\sum_{I\in\mathcal{L}}\Delta_{\ci{I}}f+\sum_{I\in\mathfrak{A}_{-\infty}^{0, \textup{fin}}}\mathbb{E}_{\ci{I}}f \ \ \ \ a.e.~and~in~L^p.\label{eq1}
\end{align}
\end{proposition}

\begin{proof}
One can easily see that
$$\sum_{I\in\mathcal{L}, m\leq\textup{rk}(I)<n}\Delta_{\ci{I}}f=\sum_{m<k\leq n}\Delta_kf=\mathbb{E}_nf-\mathbb{E}_mf.$$
\begin{itemize}
\item $\mathbb{E}_nf\rightarrow f$ a.e. as $n\rightarrow \infty$.\\
For each $x\in\mathbb{R}$, consider the unique interval $I_k(x)$ in $\mathcal{L}_k$ containing $x$, and let $I(x)=\bigcap_{k\in\mathbb{Z}}I_k(x)\in\mathfrak{A}_\infty$.\\
If $|I(x)|=0$, then by Lebesgue Differentiation Theorem, 
$$\lim_{n\rightarrow\infty}\mathbb{E}_nf(x)=\lim_{n\rightarrow\infty}\frac{1}{|I_n(x)|}\int_{I_n(x)}f=f(x)$$
holds almost everywhere for such $x$.\\
If $|I(x)|>0$, then since $f$ is $\mathfrak{A}_\infty$-measurable, we have pointwisely,
$$\lim_{n\rightarrow\infty}\mathbb{E}_nf(x)=\lim_{n\rightarrow\infty}\frac{1}{|I_n(x)|}\int_{I_n(x)}f=\frac{1}{|I(x)|}\int_{I(x)}f=f(x).$$
\item $\mathbb{E}_mf\rightarrow \sum_{I\in\mathfrak{A}_{-\infty}^{0, \textup{fin}}}\mathbb{E}_{\ci{I}}f$ as $m \rightarrow -\infty$.\\
For each $x\in\mathbb{R}$, take now $I(x)=\bigcup_{k\in\mathbb{Z}}I_k(x)$.\\
If $|I(x)|=\infty$, then since $f\in L^p$, we can estimate
$$|\mathbb{E}_mf(x)|\leq\left[\frac{1}{|I_m(x)|}\int_{I_m(x)}|f|^p\right]^\frac{1}{p}\leq|I_m(x)|^{-\frac{1}{p}}||f||_p\rightarrow0$$
as $m\rightarrow-\infty$.\\
If $I(x)\in\mathfrak{A}_{-\infty}^{0,\textup{fin}}$, then pointwisely,
$$\lim_{m\rightarrow-\infty}\mathbb{E}_mf(x)=\lim_{m\rightarrow-\infty}\frac{1}{|I_m(x)|}\int_{I_m(x)}f=\frac{1}{|I(x)|}\int_{I(x)}f=\mathbb{E}_{\ci{I(x)}}f(x).$$
\item $f=\sum_{I\in\mathcal{L}}\Delta_{\ci{I}}f+\sum_{I\in\mathfrak{A}_{-\infty}^{0, \textup{fin}}}\mathbb{E}_{\ci{I}}f$ in $L^p$.\\
We have proved the almost everywhere convergence. Note that
$$|\mathbb{E}_nf-f|\leq 2f^*~~~\ \ \textup{and}\ \ \ ~~~\left|\mathbb{E}_mf-\sum_{I\in\mathfrak{A}_{-\infty}^{0, \textup{fin}}}\mathbb{E}_{\ci{I}}f\right|\leq 2f^*,$$
where $f^*$ is the Hardy-Littlewood maximal funtion of $f$. Hardy-Littlewood Maximal Inequality and Dominated Convergence Theorem imply the $L^p$ convergence.
\end{itemize}
\end{proof}

\begin{definition}
The \emph{maximal function} is defined to be 
$$Mf(x)=\sup_{x\in I, I\in\mathcal{L}}|\mathbb{E}_{\ci{I}}f|=\sup_{x\in I, I\in\mathcal{L}}\frac{1}{|I|}\left|\int_If\right|.$$\\
We say that $f\in H_*^1$ if $Mf\in L^1$ with norm $||f||_{H_*^1}=||Mf||_1$.\\
The \emph{square function} is defined to be
$$Sf(x)=\left[\sum_{I\in\mathcal{L}}|\Delta_{\ci{I}}f|^2+\sum_{I\in\mathfrak{A}_{-\infty}^{0, \textup{fin}}}|\mathbb{E}_{\ci{I}}f|^2\right]^\frac{1}{2}.$$\\
We say that $f\in H_S^1$ if $Sf\in L^1$ with norm $||f||_{H_S^1}=||Sf||_1$.\\
Both $H_*^1$ and $H_S^1$ are called \emph{martingale Hardy spaces}.  
\end{definition}

Now we are ready to state our main theorem:

\begin{theorem}[B. J. Davis \cite{BJD}] $||Mf||_1\approx||Sf||_1$  or $H_*^1=H_S^1$.\label{thm1} 
\end{theorem}

This theorem and its proof can be found for example in \cite{BJD}, \cite{AMG}, and \cite{AMG1}, all using probabilistic method. We will give a relatively easier analysis proof here, in the spirit of \cite{AMG1}, but with the help of Carleson sequence and balayage, which is also interesting by itself.

\section{Carleson sequence and balayage}
This section is devoted to investigating the relation between Carleson sequence and balayage. We start with the following definitions:

\begin{definition}
A sequence of non-negative numbers $\{a_{\ci{I}}\}_{\ci{I\in\mathcal{L}}}$ is a \emph{Carleson sequence}, if there exists a constant $C>0$ such that
$$\frac{1}{|I|}\sum_{J\in\mathcal{L}, J\subseteq I}a_{\ci{J}}\leq C$$
for each $I\in\mathcal{L}\bigcup\mathfrak{A}_{-\infty}^{0, \textup{fin}}$. Moreover, we denote the infimun of such constant by Carl($a_{\ci{I}}$).
\end{definition}

\begin{definition}
Given a  sequence of numbers $\{a_{\ci{I}}\}_{\ci{I\in\mathcal{L}}}$, its \emph{balayage} is defined to be the function
$$g(x)=\sum_{I\in\mathcal{L}}\frac{1}{|I|}\mathbf{1}_{\ci{I}}a_{\ci{I}}.$$
\end{definition}

\begin{definition}\label{Def BMO}
A locally integrable function $g(x)$ on $\mathbb{R}$ is in the class of \emph{martingale BMO}, if
\begin{itemize}
\item $\left[\frac{1}{|I|}\displaystyle\int_I\sum_{J\in\mathcal{L}, J\subseteq I}|\Delta_{\ci{J}}g|^2\right]^{\frac{1}{2}}=C_1<\infty$, for all $I\in\mathcal{L}$.
\item $||\Delta_{\ci{I}}g||_\infty=C_2<\infty$, for all $I\in\mathcal{L}$.
\end{itemize}
Moreover, we define $||g||\ti{BMO}=\max\{C_1, C_2\}.$
\end{definition}

\begin{remark}
As in the classical case, martingale BMO space does not distinguish constant functions. In other words, we should think of martingale BMO space as the quotient of the above class by the space of $\{g\in L_{loc}^1: \mathbb{E}_{\ci{I}}g$ \textup{is constant}, \textup{for al}l $I\in\mathfrak{A}_{-\infty}^0\}$. \label{rmk}
\end{remark}

\begin{remark} This definition of martingale BMO functions is stronger than the classical one in harmonic analysis, since by (\ref{eq1}) and othogonality of the difference operators,
\begin{align}
\frac{1}{|I|}\int_I|g(x)-\langle g\rangle_{\ci{I}}|^2=\frac{1}{|I|}\int_I \lim_{n\rightarrow\infty}|\sum_{J\in\mathcal{L}, J\subseteq I, \textup{rk}(I)<n}\Delta_{\ci{J}}g|^2
=\frac{1}{|I|}\int_I\sum_{J\in\mathcal{L}, J\subseteq I}|\Delta_{\ci{J}}g|^2,\label{eq2}
\end{align}
where $\langle g\rangle_{\ci{I}}=\frac{1}{|I|}\int_Ig$ is the average of $g(x)$ over $I$.\\
It was known for a long time that the additional condition is needed. One can look, for example, at \cite{AMG} in detail.
\end{remark}

Now we have the precise statement:

\begin{theorem}
\textup{(Carleson sequence and balayage)}\\
\begin{enumerate}
\item Given a Carleson sequence $\{|a_{\ci{I}}|\}_{\ci{I\in\mathcal{L}}}$, the following inequality holds
\begin{align}
\left|\left|g(x)=\sum_{I\in\mathcal{L}}\frac{1}{|I|}\mathbf{1}_{\ci{I}}a_{\ci{I}}\right|\right|\ti{BMO}\leq2\textup{Carl}(|a_{\ci{I}}|).\label{eq3}
\end{align}
\item For each function $g(x)$ in the class of martingale BMO, there exists $\phi(x)\in L^\infty$, and a Carleson sequence $\{|a_{\ci{I}}|\}_{\ci{I\in\mathcal{L}}}$, such that
\begin{align}
g=\phi+\sum_{I\in\mathcal{L}}\frac{1}{|I|}\mathbf{1}_{\ci{I}}a_{\ci{I}},\label{eq4}
\end{align}
where $||\phi||_\infty\leq2||g||\ti{BMO}$, and $\textup{Carl}(|a_{\ci{I}}|)\leq3||g||\ti{BMO}$.
\end{enumerate}\label{thm2}
\end{theorem}

\begin{remark}
One might hope that the reverse inequality to (\ref{eq3})  holds in the strict sense, but this is not the case due to \cite{SPAV}. Nevertheless, we still have (\ref{eq4}) as a partial reverse answer, which is good enough for our purpose.
\end{remark}

\begin{proof} 
\begin{enumerate}
\item 
Check by definition\\
\begin{align*}
\frac{1}{|I|}\int_I\sum_{J\in\mathcal{L}, J\subseteq I}|\Delta_{\ci{J}}g|^2
& =\frac{1}{|I|}\int_I|g(x)-\langle g\rangle_{\ci{I}}|^2 \ \ \ \ (\ref{eq2})\\
& =\langle g^2\rangle_{\ci{I}}-\langle g\rangle_{\ci{I}}^2 \\
&=\frac{1}{|I|}\int_I|\sum_{J\in\mathcal{L}}\frac{1}{|J|}\mathbf{1}_{\ci{J}}a_{\ci{J}}|^2-(\frac{1}{|I|}\int_I\sum_{J\in\mathcal{L}}\frac{1}{|J|}\mathbf{1}_{\ci{J}}a_{\ci{J}})^2\\
&=\frac{1}{|I|}\int_I|\sum_{J\in\mathcal{L},J\subseteq I}\frac{1}{|J|}\mathbf{1}_{\ci{J}}a_{\ci{J}}|^2-(\frac{1}{|I|}\int_I\sum_{J\in\mathcal{L},J\subseteq I}\frac{1}{|J|}\mathbf{1}_{\ci{J}}a_{\ci{J}})^2\\
& \leq\frac{1}{|I|}\int_I(\sum_{J\in\mathcal{L},J\subseteq I}\frac{1}{|J|}\mathbf{1}_{\ci{J}}|a_{\ci{J}}|)^2\\
& =\frac{1}{|I|}\int_I\sum_{J, K\subseteq I}\frac{1}{|J|}\frac{1}{|K|}|a_{\ci{J}}||a_{\ci{K}}|\mathbf{1}_{\ci{J}}\mathbf{1}_{\ci{K}}\\
& =\frac{1}{|I|}\sum_{J\subseteq I}|a_{\ci{J}}|\frac{1}{|J|}\sum_{K\subseteq J}|a_{\ci{K}}|+\frac{1}{|I|}\sum_{K\subseteq I}|a_{\ci{K}}|\frac{1}{|K|}\sum_{J\subsetneqq K}|a_{\ci{J}}|\\
& \leq \frac{1}{|I|}\sum_{J\subseteq I}|a_{\ci{J}}| \textup{Carl}(|a_{\ci{I}}|)+\frac{1}{|I|}\sum_{K\subseteq I}|a_{\ci{K}}|\textup{Carl}(|a_{\ci{I}}|)\\
& \leq 2\textup{Carl}(|a_{\ci{I}}|)^2,
\end{align*}
holds for any $I\in\mathcal{L}$, where the forth equality follows from
$$\sum_{J\in\mathcal{L},J\supsetneqq I}\frac{1}{|J|}\mathbf{1}_{\ci{J}}a_{\ci{J}}$$
being a constant on $I$.\\
Moreover, we can compute that for any $I\in\mathcal{L}$,
\begin{align*}
\mathbb{E}_{\ci{I}}g
& =\langle g\rangle_{\ci{I}}\mathbf{1}_{\ci{I}}=(\frac{1}{|I|}\int_I\sum_{k\in\mathcal{L}}\frac{1}{|K|}\mathbf{1}_{\ci{K}}a_{\ci{K}})\mathbf{1}_{\ci{I}}\\
& =(\frac{1}{|I|}\sum_{K\in\mathcal{L},K\subseteq I}a_{\ci{K}})\mathbf{1}_{\ci{I}}+(\sum_{K\in\mathcal{L},K\supsetneqq I}\frac{1}{|K|}a_{\ci{K}})\mathbf{1}_{\ci{I}}.
\end{align*}
Thus, we can estimate
\begin{align*}
|\Delta_{\ci{I}}g| 
& =|-\mathbb{E}_{\ci{I}}g+\sum_{J\in\textup{child}(I)}\mathbb{E}_{\ci{J}}g|\\
& =|\sum_{J\in\textup{child}(I)}(\frac{1}{|J|}\sum_{K\in\mathcal{L},K\subseteq J}a_{\ci{K}})\mathbf{1}_{\ci{J}}+\sum_{J\in\textup{child}(I)}(\sum_{K\in\mathcal{L},K\supsetneqq J}\frac{1}{|K|}a_{\ci{K}})\mathbf{1}_{\ci{J}}\\
& \ \ \ \ -(\frac{1}{|I|}\sum_{K\in\mathcal{L},K\subseteq I}a_{\ci{K}})\mathbf{1}_{\ci{I}}-(\sum_{K\in\mathcal{L},K\supsetneqq I}\frac{1}{|K|}a_{\ci{K}})\mathbf{1}_{\ci{I}}|\\
& =|\sum_{J\in\textup{child}(I)}(\frac{1}{|J|}\sum_{K\in\mathcal{L},K\subseteq J}a_{\ci{K}})\mathbf{1}_{\ci{J}}-(\frac{1}{|I|}\sum_{K\in\mathcal{L},K\subsetneqq I}a_{\ci{K}})\mathbf{1}_{\ci{I}}|\\
& \leq\sum_{J\in\textup{child}(I)}(\frac{1}{|J|}\sum_{K\in\mathcal{L},K\subseteq J}|a_{\ci{K}}|)\mathbf{1}_{\ci{J}}+(\frac{1}{|I|}\sum_{K\in\mathcal{L},K\subsetneqq I}|a_{\ci{K}}|)\mathbf{1}_{\ci{I}}\leq2\textup{Carl}(|a_{\ci{I}}|).
\end{align*}

\item
The proof of this statement is a modification of the proof of John-Nirenberg theorem which is suggested in \cite{JBG}.
\par Given a martingale BMO function $g(x)$, by (\ref{eq2}) we have
$$\sup_{I\in\mathcal{L}}\frac{1}{|I|}\int_I|g(x)-\langle g\rangle_{\ci{I}}|\leq\sup_{I\in\mathcal{L}}\left[\frac{1}{|I|}\int_I|g(x)-\langle g\rangle_{\ci{I}}|^2\right]^\frac{1}{2}=||g||\ti{BMO}.$$
\par Fix the 0-th generation $\mathcal{L}_0$ of the lattice $\mathcal{L}$, recall that $\mathcal{L}_0$ is a collection of non-trivial, finite, disjoint intervals covering $\mathbb{R}$, so we can write $\mathcal{L}_0=\{I_m\}_{m\in\mathfrak{M}}$. 
\par Over each interval $I_m, m\in\mathfrak{M}$, let us first do the Calder\'{o}n-Zygmund decomposition to the function $|g(x)-\langle g\rangle_{\ci{I_m}}|$ at height $\lambda=2||g||\ti{BMO}$ with respect to the lattice $\mathcal{L}$, from this we obtain:\\
A collection of disjoint intervals $\{I_j^{(1)}\}\subseteq\mathcal{L}$ contained in $I_m$, such that over each $I_j^{(1)}$,
$$\frac{1}{|I_j^{(1)}|}\int_{I_j^{(1)}}|g(x)-\langle g\rangle_{\ci{I_m}}|>2||g||\ti{BMO}.$$
This immediately implies
\begin{itemize}
\item $|g(x)-\langle g\rangle_{\ci{I_m}}|\leq2||g||\ti{BMO}$ a.e. over $I_m\setminus\bigcup_jI_j^{(1)}.$\\
For each $x\in I_m\setminus\bigcup_jI_j^{(1)}$, from the Calder\'{o}n-Zygmund decomposition, we know that over all intervals $\{I_k(x)\}_{k\geq1}$,
$$\frac{1}{|I_k(x)|}\int_{I_k(x)}|g(x)-\langle g\rangle_{\ci{I_m}}|\leq2||g||\ti{BMO},$$
where $I_k(x)$ is the unique interval in $\mathcal{L}_k, k\geq 1$ containing $x$.\\
If $I(x)=\bigcap_{k\geq 1}I_k(x)\in\mathfrak{A}_\infty$ has $|I(x)|=0$, then by Lebesgue Differentiation Theorem, we can conclude that
$$|g(x)-\langle g\rangle_{\ci{I_m}}|
=\lim_{k\rightarrow\infty}\frac{1}{|I_k(x)|}\int_{I_k(x)}|g(x)-\langle g\rangle_{\ci{I_m}}|\leq2||g||\ti{BMO}$$
holds almost everywhere for such $x$. \\
If $|I(x)|>0$, then since $g(x)$ is $\mathfrak{A}_\infty$-measurable, we can conclude that
\begin{align*}
|g(x)-\langle g\rangle_{\ci{I_m}}| 
& =\frac{1}{|I(x)|}\int_{I(x)}|g(x)-\langle g\rangle_{\ci{I_m}}|\\
& =\lim_{k\rightarrow\infty}\frac{1}{|I_k(x)|}\int_{I_k(x)}|g(x)-\langle g\rangle_{\ci{I_m}}|\leq2||g||\ti{BMO}.
\end{align*}
\item $\sum_j|I_j^{(1)}|\leq\frac{1}{2}|I_m|.$
\begin{align*}
\sum_j|I_j^{(1)}|
& \leq\sum_j\frac{1}{2||g||\ti{BMO}}\int_{I_j^{(1)}}|g(x)-\langle g\rangle_{\ci{I_m}}|\\
& \leq\frac{1}{2||g||\ti{BMO}}\int_{I_m}|g(x)-\langle g\rangle_{\ci{I_m}}|\leq\frac{1}{2}|I_m|.
\end{align*}
\end{itemize}
\par Next, let us do the Calder\'{o}n-Zygmund decomposition over each $I_j^{(1)}$ to the function $|g(x)-\langle g\rangle_{I_j^{(1)}}|$ at again height $\lambda=2||g||\ti{BMO}$ with respect to the lattice $\mathcal{L}$, from this we obtain
\par A collection of disjoint intervals $\{I_k^{(2)}\}\subseteq\mathcal{L}$ contained in $I_j^{(1)}$, such that over each $I_k^{(2)}$,
$$\frac{1}{|I_k^{(2)}|}\int_{I_k^{(2)}}|g(x)-\langle g\rangle_{\ci{I_j^{(1)}}}|>2||g||\ti{BMO}.$$
Similar argument yields
\begin{itemize}
\item $|g(x)-\langle g\rangle_{\ci{I_j^{(1)}}}|\leq2||g||\ti{BMO}$ a.e. over $I_j^{(1)}\setminus\bigcup_kI_k^{(2)}$.
\item $\sum_k|I_k^{(2)}|\leq(\frac{1}{2})|I_j^{(1)}|$.\\
Moreover, if we sum over all intervals $\{I_j^{(1)}\}$, we will have\\
$\sum_k|I_k^{(2)}|\leq\frac{1}{2}\sum_j|I_j^{(1)}|\leq(\frac{1}{2})^2|I_m|$.
\end{itemize}
\par Continue this process indefinitely. At stage $n$ we get disjoint intervals $\{I_k^{(n)}\}\subseteq\mathcal{L}$, such that over each $I_k^{(n)}$,\\
\begin{itemize}
\item $|g(x)-\langle g\rangle_{\ci{I_j^{(n-1)}}}|\leq2||g||\ti{BMO}$ a.e. over $I_j^{(n-1)}\setminus\bigcup_kI_k^{(n)}$.
\item $\sum_k|I_k^{(n)}|\leq(\frac{1}{2})^n|I_m|$.
\end{itemize}
\par Now take $\mathcal{G}_m=\bigcup_{j, n}\{I_j^{(n)}\}\cup\{I_m\}$, and let $\mathcal{G}=\bigcup_m\mathcal{G}_m\subseteq\mathcal{L}$, define $\{a_{\ci{I}}\}_{I\in\mathcal{L}}$ to be
$$a_{\ci{I_k^{(n)}}}=|I_k^{(n)}|(\langle g\rangle_{\ci{I_k^{(n)}}}-\langle g\rangle_{\ci{I_j^{(n-1)}}}),$$
and for $I\notin\mathcal{G}$, simply set
$$a_{\ci{I}}=0.$$
\par Finally, define the function $\phi(x)=g(x)-\sum_{I\in\mathcal{L}}\frac{1}{|I|}\mathbf{1}_{\ci{I}}a_{\ci{I}}$.\\
\\ \emph{Claim 1}: $||\phi||_\infty\leq2||g||\ti{BMO}$.\\
\par If there exists $x\in I_m, I_m\in\mathcal{L}_0$, such that at all stages $n\geq 1$, one can find an interval $I_k^{(n)}$ containing $x$, let $I(x)\in\mathfrak{A}_\infty$ be the intersection of all those intervals, then by construction
$$|\phi(x)|=|g(x)-\sum_{I\in\mathcal{L}}\frac{1}{|I|}\mathbf{1}_{\ci{I}}a_{\ci{I}}|
=|g(x)-\lim_{n\rightarrow\infty}\langle g\rangle_{\ci{I_k^{(n)}}}|=|g(x)-\langle g\rangle_{\ci{I(x)}}|=0.$$
Otherwise, for arbitrary $I_m\in\mathcal{L}_0$, we have
$$|\phi(x)|=|g(x)-\sum_{I\in\mathcal{L}}\frac{1}{|I|}\mathbf{1}_{\ci{I}}a_{\ci{I}}|=|g(x)-\langle g\rangle_{\ci{I_j^{(n-1)}}}|\leq 2||g||\ti{BMO}$$
almost everywhere over $I_j^{(n-1)}\setminus\bigcup_kI_k^{(n)}$ for all stages $n\geq 1$, therefore $||\phi||_\infty\leq2||g||\ti{BMO}$.\\
\\ \emph{Claim 2}: $\textup{Carl}(|a_I|)\leq3||g||\ti{BMO}$.\\
\par Fix an interval $J\in\mathcal{L}\bigcup\mathfrak{A}_{-\infty}^{0,\textup{fin}}$, we have
$$\frac{1}{|J|}\sum_{I\in\mathcal{L},I\subseteq J}|a_{\ci{I}}|=\frac{1}{|J|}\sum_{I_k^{(n)}\subseteq J}|I_k^{(n)}||\langle g\rangle_{\ci{I_k^{(n)}}}-\langle g\rangle_{\ci{I_j^{(n-1)}}}|.$$
\par Note that $g\in BMO$ gives $||\Delta_{\ci{I}}g||_\infty\leq||g||\ti{BMO}$, for all $I\in\mathcal{L}$, so 
\begin{align*}
|\langle g\rangle_{\ci{I_k^{(n)}}}-\langle g\rangle_{\ci{I_j^{(n-1)}}}|
& \leq|\langle g\rangle_{\ci{I_k^{(n)}}}-\langle g\rangle_{\ci{\widetilde{I_k^{(n)}}}}|+|\langle g\rangle_{\ci{\widetilde{I_k^{(n)}}}}-\langle g\rangle_{\ci{I_j^{(n-1)}}}|\\
& \leq||\Delta_{\ci{\widetilde{I_k^{(n)}}}}g||_\infty+\frac{1}{|\widetilde{I_k^{(n)}}|}\int_{\widetilde{I_k^{(n)}}}|g(x)-\langle g\rangle_{\ci{I_j^{(n-1)}}}|\\
& \leq||g||\ti{BMO}+2||g||\ti{BMO}=3||g||\ti{BMO},
\end{align*}
where $\widetilde{I_k^{(n)}}$ is the parent of $I_k^{(n)}$.
\par To estimate $\frac{1}{|J|}\sum_{I_k^{(n)}\subseteq J}|I_k^{(n)}|$, by our construction, it suffices to assume $J$ being contained in some $I_m\in\mathcal{L}_0$, and
$$\frac{1}{|J|}\sum_{I_k^{(n)}\subseteq J}|I_k^{(n)}|=\frac{1}{|J|}\sum_{n\geq 1}\sum_k|I_k^{(n)}|\leq\frac{1}{|J|}\sum_{n\geq 1}(\frac{1}{2})^n|J|=1.$$
\par Hence, we obtain $\frac{1}{|J|}\sum_{I\in\mathcal{L},I\subseteq J}|a_I|\leq 3||g||\ti{BMO}$, for all  $J\in\mathcal{L}\bigcup\mathfrak{A}_{-\infty}^{0,\textup{fin}}$, which completes the proof. 
\end{enumerate}
\end{proof}

\section{Proof of the main theorem}
We are now ready to prove \textbf{Theorem \ref{thm1}} using \textbf{Theorem \ref{thm2}}.
\begin{proof}
\begin{enumerate} 
\item $||Mf||_1\lesssim||Sf||_1$ or $H_*^1\supseteq H_S^1$
\par To prove this half of the theorem, we need the following celebrated Fefferman's inequality \cite{CF} whose proof will be postponed to the next section: 
\begin{theorem}
For $f\in H_S^1$ and $g\in BMO$, we have 
\begin{align}
\int fg\leq 2||f||_{H_S^1}||g||\ti{BMO}.\label{eq5}
\end{align}\label{thm3}
\end{theorem}
Now, define the 'cut-off' of $Mf(x)$ by
$$M_nf(x)=\max_{x\in I, I\in\bigcup_{k=-n}^n\mathcal{L}_k}|\langle f\rangle_{\ci{I}}|,$$ 
so obviously, $M_nf$ increases to $Mf$ pointwisely. Moreover, for any $J\in\mathcal{L}_n$, we have $M_nf|_{\ci{J}}$ must be constant. In addition, define
$$I(J)=\{I\supseteq J, I~\textup{is maximal in} \bigcup_{k=-n}^n\mathcal{L}_k: M_nf|_{\ci{J}}=|\langle f\rangle_{\ci{I}}|\},$$
and let $a_{\ci{I(J)}}$ be such that $a_{\ci{I(J)}}\langle f\rangle_{\ci{I(J)}}=\int_JM_nf$, which implies $|a_{\ci{I(J)}}|=|J|$. Finally define
$$a_{\ci{I}}=\sum_{\{J:I(J)=I\}}a_{\ci{I(J)}},$$
therefore, we have
$$\langle f\rangle_{\ci{I}}a_{\ci{I}} =\langle f\rangle_{\ci{I}}\left[\sum_{\{J:I(J)=I\}}a_{\ci{I(J)}}\right]
= \sum_{\{J:I(J)=I\}}\langle f\rangle_{\ci{I(J)}}a_{\ci{I(J)}}
= \sum_{\{J:I(J)=I\}}\int_JM_nf.$$
\par Summing over all intervals $I\in\bigcup_{k=-n}^n\mathcal{L}_k$, we obtain
$$\sum_{I\in\bigcup_{k=-n}^n\mathcal{L}_k}\langle f\rangle_{\ci{I}}a_{\ci{I}}=\sum_{I\in\bigcup_{k=-n}^n\mathcal{L}_k}\sum_{\{J:I(J)=I\}}\int_JM_nf=\int M_nf.$$
\par Having this, we can write
$$\int M_nf=\sum_{I\in\bigcup_{k=-n}^n\mathcal{L}_k}\langle f\rangle_{\ci{I}}a_{\ci{I}}=\int f\sum_{I\in\bigcup_{k=-n}^n\mathcal{L}_k}\frac{1}{|I|}\mathbf{1}_{\ci{I}}a_{\ci{I}}=\int fg,$$
where $g(x)=\sum_{I\in\bigcup_{k=-n}^n\mathcal{L}_k}\frac{1}{|I|}\mathbf{1}_{\ci{I}}a_{\ci{I}}$.\\
\\ \emph{Claim}: $||g||\ti{BMO}\leq 2.$\\
\par Let us first consider $\textup{Carl}(|a_{\ci{I}}|)$ for the sequence $\{|a_{\ci{I}}|\}$, such that $I\in\bigcup_{k=-n}^n\mathcal{L}_k$. Take any $K\in\bigcup_{k=-n}^n\mathcal{L}_k$, we have
$$\sum_{I\subseteq K}|a_{\ci{I}}| = \sum_{I\subseteq K}|\sum_{\{J:I(J)=I\}}a_{\ci{I(J)}}|
\leq \sum_{J\in\mathcal{L}_n,J\subseteq K}|a_{\ci{I(J)}}|
= \sum_{J\in\mathcal{L}_n,J\subseteq K}|J|\leq|K|, $$
so $\textup{Carl}(|a_I|)\leq 1$, and by (\ref{eq3}), we know $||g||\ti{BMO}\leq2$.
\par To complete our proof, by (\ref{eq5}), we can conclude
$$\int M_nf=\int fg\leq 2||f||_{H_S^1}||g||\ti{BMO}\leq 4||f||_{H_S^1}.$$
Letting $n\rightarrow\infty$, by Monotone Convergence Theorem, we obtain 
$$||Mf||_1\leq 4||Sf||_1.$$

\item $||Mf||_1\gtrsim ||Sf||_1$ or $H_*^1\subseteq H_S^1$
\par To prove this half of the theorem, we need another auxilary lemma which will also be proved in the next section:
\begin{lemma}
For $f\in H_*^1$ and a Carleson sequence $\{|a_{\ci{I}}|\}_{\ci{I\in\mathcal{L}}}$, we have
\begin{align}
|\sum_{I\in\mathcal{L}}\langle f\rangle_{\ci{I}}a_{\ci{I}}|\leq ||f||_{H_*^1}\textup{Carl}(|a_{\ci{I}}|).\label{eq6}
\end{align}\label{L1}
\end{lemma}
By (\ref{eq4}) and (\ref{eq6}), we can conclude that fix a function $f\in H_*^1$, and for any $g\in BMO$, we have
\begin{align*}
|\int fg| &= |\int f(\phi+\sum_{I\in\mathcal{L}}\frac{1}{|I|}\mathbf{1}_{\ci{I}}a_{\ci{I}})|\leq |\int f\phi|+|\int f\sum_{I\in\mathcal{L}}\frac{1}{|I|}\mathbf{1}_{\ci{I}}a_{\ci{I}}|\\
&\leq |\int Mf\phi|+|\sum_{I\in\mathcal{L}}\langle f\rangle_{\ci{I}}a_{\ci{I}}|\leq ||Mf||_1||\phi||_\infty+||f||_{H_*^1}\textup{Carl}(a_{\ci{I}})\\
&\leq ||f||_{H_*^1}2||g||\ti{BMO}+||f||_{H_*^1}3||g||\ti{BMO}=5||f||_{H_*^1}||g||\ti{BMO}.
\end{align*}
\emph{Claim}: The above estimation implies $||Sf||_1\lesssim||Mf||_1$.\\
\par It suffices to prove this claim on a dense set of  funtions $f(x)$ for which the corresponding sequence $\{\Delta_{\ci{I}}f\}_{\ci{I\in\mathcal{L}}}$ and $\{\mathbb{E}_{\ci{I}}f\}_{\ci{I\in\mathfrak{A}_{-\infty}^{0,\textup{fin}}}}$ have only finitely many non-zero terms.
\par Let us construct a function $g\in BMO$ from the function $f$, such that $\int Sf=\int fg$.
\par First take $g_{\ci{I}}=\Delta_{\ci{I}}f/Sf$ for $I\in\mathcal{L}$, and take $g_{\ci{I}}=\mathbb{E}_{\ci{I}}f/Sf$ for $I\in\mathfrak{A}_{-\infty}^{0,\textup{fin}}$, thus $\sum_{I\in\mathcal{L}}|g_{\ci{I}}|^2+\sum_{I\in\mathfrak{A}_{-\infty}^{0,\textup{fin}}}|g_{\ci{I}}|^2=1$, and $\sum_{I\in\mathcal{L}}\Delta_{\ci{I}}fg_{\ci{I}}+\sum_{I\in\mathfrak{A}_{-\infty}^{0,\textup{fin}}}\mathbb{E}_{\ci{I}}fg_{\ci{I}}=Sf$. One can also check easily that
\begin{align*}
\int Sf
& =\sum_{I\in\mathcal{L}}\int\Delta_{\ci{I}}fg_{\ci{I}}+\sum_{I\in\mathfrak{A}_{-\infty}^{0,\textup{fin}}}\int\mathbb{E}_{\ci{I}}fg_{\ci{I}}\\
& =\sum_{I\in\mathcal{L}}\int\Delta_{\ci{I}}f\Delta_{\ci{I}}g_{\ci{I}}+\sum_{I\in\mathfrak{A}_{-\infty}^{0,\textup{fin}}}\int\mathbb{E}_{\ci{I}}f\mathbb{E}_{\ci{I}}g_{\ci{I}}.
\end{align*}
\par Consider now $g=\sum_{I\in\mathcal{L}}\Delta_{\ci{I}}g_{\ci{I}}+\sum_{I\in\mathfrak{A}_{-\infty}^{0,\textup{fin}}}\mathbb{E}_{\ci{I}}g_{\ci{I}}$, since $\{\Delta_{\ci{I}}f\}_{\ci{I\in\mathcal{L}}}$ and $\{\mathbb{E}_{\ci{I}}f\}_{\ci{I\in\mathfrak{A}_{-\infty}^{0,\textup{fin}}}}$ have only finitely many non-zero terms, this is actually a finite sum which is obviously well-defined. Moreover, it is easy to compute that $\Delta_{\ci{I}}g=\Delta_{\ci{I}}g_{\ci{I}}$ for all $I\in\mathcal{L}$, and $\mathbb{E}_{\ci{I}}g=\mathbb{E}_{\ci{I}}g_{\ci{I}}$ for all $\mathfrak{A}_{-\infty}^{0,\textup{fin}}$. Let us show by definition that $g\in BMO$.
\par For any interval $I\in\mathcal{L}$, $|\Delta_{\ci{I}}g|=|\Delta_{\ci{I}}g_{\ci{I}}|\leq2g_{\ci{I}}^*$ where $g_{\ci{I}}^*$ is the Hardy-Littlewood maximal function of $g_{\ci{I}}$. Note that $||g_{\ci{I}}||_\infty\leq1$ for all $I\in\mathcal{L}$, thus we always have $g_{\ci{I}}^*\leq1$ and so $|\Delta_{\ci{I}}g|\leq2$.
\par For any interval $I\in\mathcal{L}$, by Hardy-Littlewood Maximal Inequality, we can estimate
$$\int_I\sum_{J\subseteq I}|\Delta_{\ci{J}}g|^2 
= \int_I\sum_{J\subseteq I}|\Delta_{\ci{J}}g_{\ci{J}}|^2\\ 
\leq 4\int_I\sum_{J\subseteq I}|g_{\ci{I}}^*|^2\\
\leq C\int_I\sum_{J\subseteq I}|g_{\ci{I}}|^2\\
\leq C|I|.$$
\par Hence, we find a funtion $g\in BMO$ such that
\begin{align*}
\int Sf
&= \sum_{I\in\mathcal{L}}\int\Delta_{\ci{I}}f\Delta_{\ci{I}}g_{\ci{I}}+\sum_{I\in\mathfrak{A}_{-\infty}^{0,\textup{fin}}}\mathbb{E}_{\ci{I}}f\mathbb{E}_{\ci{I}}g_{\ci{I}}\\
&= \sum_{I\in\mathcal{L}}\int\Delta_{\ci{I}}f\Delta_{\ci{I}}g+\sum_{I\in\mathfrak{A}_{-\infty}^{0,\textup{fin}}}\mathbb{E}_{\ci{I}}f\mathbb{E}_{\ci{I}}g\\
&=\int(\sum_{I\in\mathcal{L}}\Delta_{\ci{I}}f+\sum_{I\in\mathfrak{A}_{-\infty}^{0,\textup{fin}}}\mathbb{E}_{\ci{I}}f)(\sum_{I\in\mathcal{L}}\Delta_{\ci{I}}g+\sum_{I\in\mathfrak{A}_{-\infty}^{0,\textup{fin}}}\mathbb{E}_{\ci{I}}g)=\int fg.
\end{align*}
Finally, we can conclude
$$||f||_{H_S^1}=\int Sf=\int fg\leq  5||f||_{H_*^1}||g||\ti{BMO}\lesssim ||f||_{H_*^1}.$$
\end{enumerate}
\end{proof}

\section{Proof of the auxilary lemmas}
In this section, we will prove \textbf{Theorem \ref{thm3}} and \textbf{Lemma \ref{L1}}, and hence complete the whole proof.
\par To show Fefferman's inequality, we take the elegant proof from \cite{AMG1}.
\begin{proof} \textbf{Theorem \ref{thm3}}
\par For $f\in H_S^1$ and $g\in BMO$, we can assume by Remark \ref{rmk} that $\mathbb{E}_{\ci{I}}g=0$, for all $I\in\mathfrak{A}_{-\infty}^0$, and so we can compute
$$|\int fg|=|\int(\sum_{I\in\mathcal{L}}\Delta_{\ci{I}}f)(\sum_{I\in\mathcal{L}}\Delta_{\ci{I}}g)|=|\sum_{I\in\mathcal{L}}\int\Delta_{\ci{I}}f\Delta_{\ci{I}}g|=|\sum_{k\in\mathbb{Z}}\int\Delta_kf\Delta_kg|.$$
\par Now, let us define $S_nf(x)=\left[\sum_{k\leq n}|\Delta_kf|^2\right]^{\frac{1}{2}}=\left[\sum_{\textup{rk}(I)\leq n-1}|\Delta_{\ci{I}}f|^2\right]^{\frac{1}{2}}$, then $0\leq S_nf(x)\leq S_{n+1}f(x)\leq Sf(x)$ holds for all $x$. A very clever idea due to C. Herz suggests that
$$|\sum_{k\leq n}\int\Delta_kf\Delta_kg|=|\sum_{k\leq n}\int\frac{\Delta_kf}{\sqrt{S_kf}}\Delta_kg\sqrt{S_kf}|\leq\left( \sum_{k\leq n}\int\frac{|\Delta_kf|^2}{S_kf}\right)^{\frac{1}{2}}\left(\sum_{k\leq n}\int|\Delta_kg|^2S_kf\right)^{\frac{1}{2}}.$$
\par Therefore, all boil down to the following simple estimations
$$\sum_{k\leq n}\int\frac{|\Delta_kf|^2}{S_kf}=\sum_{k\leq n}\int\frac{S_k^2f-S_{k-1}^2f}{S_kf}\leq 2\sum_{k\leq n}\int (S_kf-S_{k-1}f)\leq 2\int Sf,$$
and
$$\sum_{k\leq n}\int|\Delta_kg|^2S_kf=\sum_{k\leq n}\int(\sum_{l=k}^n|\Delta_lg|^2)(S_kf-S_{k-1}f),$$
note that for every $k$, $S_kf-S_{k-1}f$ is a constant on children of intervals $I$ such that $\textup{rk}(I)=k-1$ and equals 0 outside intervals $I$, thus for an interval $J\in\textup{child}(I)$, we have
\begin{align*}
\int_J(\sum_{l=k}^n|\Delta_lg|^2)(S_kf-S_{k-1}f)
& =\int_J\left[\frac{1}{|J|}\int_J\sum_{l=k}^n|\Delta_lg|^2\right](S_kf-S_{k-1}f)\\
& \leq\int_J\left[\frac{1}{|J|}\int_J|\Delta_{\ci{I}}g|^2+\sum_{K\in\mathcal{L},K\subseteq J}|\Delta_{\ci{K}}g|^2\right](S_kf-S_{k-1}f)\\
& \leq 2||g||_{\ti{BMO}}^2\int_J(S_kf-S_{k-1}f). ~~~~~(\textup{by Definition \ref{Def BMO}})
\end{align*}
This given, we can estimate
$$\sum_{k\leq n}\int(\sum_{l=k}^n|\Delta_lg|^2)(S_kf-S_{k-1}f)\leq 2||g||_{\ti{BMO}}^2\sum_{k\leq n}\int(S_kf-S_{k-1}f)\leq 2||g||_{\ti{BMO}}^2\int Sf.$$
\end{proof}

To show the second auxilary lemma, we modify the proof using level sets comparision in \cite{ST} to get this sharper inequality.
\begin{proof} \textbf{Lemma \ref{L1}}
\par Define $E_k=\{x\in\mathbb{R}:Mf(x)>(1+\varepsilon)^k\}$, and take $\mathcal{E}_k=\{I\in\mathcal{L}:I\in E_k\}$. Note that $E_k$ is a finite disjoint union of maximal intervals in $\mathcal{E}_k$, maximal is considered in the sense of inclusion.\\
\\ \emph{Claim 1}: $|\sum_{I\in\mathcal{L}}\langle f\rangle_{\ci{I}}a_{\ci{I}}|\leq\varepsilon\sum_{k\in\mathbb{Z}}(1+\varepsilon)^k|E_k|\textup{Carl}(|a_{\ci{I}}|).$\\
\par First note that $\sum_{I\in\mathcal{E}_k}|a_{\ci{I}}|\leq\sum_{I\subseteq E_k}|a_{\ci{I}}|\leq|E_k|\textup{Carl}(|a_{\ci{I}}|)$, because $\{|a_{\ci{I}}|\}_{\ci{I\in\mathcal{L}}}$ is a Carleson squence. Using this fact, we can estimate
\begin{align*}
\varepsilon\sum_{k\in\mathbb{Z}}(1+\varepsilon)^k|E_k|\textup{Carl}(|a_{\ci{I}}|)
& \geq \varepsilon\sum_{k\in\mathbb{Z}}(1+\varepsilon)^k\sum_{I\in\mathcal{E}_k}|a_{\ci{I}}| = \varepsilon\sum_{k\in\mathbb{Z}}(1+\varepsilon)^k\sum_{l=0}^\infty\sum_{I\in\mathcal{E}_{k+l}\setminus\mathcal{E}_{k+l+1}}|a_{\ci{I}}|\\
& =\sum_{l=0}^\infty\frac{\varepsilon}{(1+\varepsilon)^{l+1}}\sum_{k\in\mathbb{Z}}(1+\varepsilon)^{k+l+1}\sum_{I\in\mathcal{E}_{k+l}\setminus\mathcal{E}_{k+l+1}}|a_{\ci{I}}|\\
& \geq\sum_{l=0}^\infty\frac{\varepsilon}{(1+\varepsilon)^{l+1}}\sum_{k\in\mathbb{Z}}\sum_{I\in\mathcal{E}_{k+l}\setminus\mathcal{E}_{k+l+1}}|\langle f\rangle_{\ci{I}}||a_{\ci{I}}|~~~~~(\textup{by definition of}~ \mathcal{E}_k)\\
&\geq\sum_{l=0}^\infty\frac{\varepsilon}{(1+\varepsilon)^{l+1}}\sum_{I\in\mathcal{L}}|\langle f\rangle_{\ci{I}}a_{\ci{I}}|=|\sum_{I\in\mathcal{L}}\langle f\rangle_{\ci{I}}a_{\ci{I}}|.
\end{align*}
\\ \emph{Claim 2}: $\int Mf\geq \varepsilon\sum_{k\in\mathbb{Z}}(1+\varepsilon)^k|E_{k+1}|.$\\
\par This is done by a commonly used trick
\begin{align*}
\int Mf 
& =\int_0^\infty|\{x:Mf>t\}|dt=\sum_{k\in\mathbb{Z}}\int_{(1+\varepsilon)^k}^{(1+\varepsilon)^{k+1}}|\{x:Mf>t\}|dt\\
& \geq \varepsilon\sum_{k\in\mathbb{Z}}(1+\varepsilon)^k|\{x:Mf>(1+\varepsilon)^{(k+1)}\}|=\varepsilon\sum_{k\in\mathbb{Z}}(1+\varepsilon)^k|E_{k+1}|.
\end{align*}
\par To finish, we combine the two claims and conclude that for any $\varepsilon>0$,
$$|\sum_{I\in\mathcal{L}}\langle f\rangle_{\ci{I}}a_{\ci{I}}|\leq(1+\varepsilon)||f||_{H_*^1}\textup{Carl}(|a_{\ci{I}}|).$$
Finally,  letting $\varepsilon\rightarrow 0$, we prove the desired inequality.
\end{proof}

\section*{Acknowledgement}
The author would like to thank his PhD thesis advisor, Serguei Treil, for suggesting this problem for his topics exam and for the invaluable guidance and support throughout the course of preparation.

\end{document}